\documentclass[10pt,draft]{amsart}
\usepackage{amsfonts}
\usepackage{amsthm}
\usepackage{amsmath}
\usepackage{graphicx}
\usepackage{amscd,amssymb,amsthm}

\newcommand{\dt}{{\rm d}t}

\newcommand{\du}{{\rm d}u}
\newcommand{\ds}{{\rm d}s}
\newcommand{\doo}{{\rm d}}

\newtheorem{theorem}{Theorem}

\title{Mills' ratio: Reciprocal convexity and functional inequalities}
\author{\'Arp\'ad Baricz}
\address{Department of Economics, Babe\c{s}-Bolyai University, Cluj-Napoca 400591,
Romania} \email{bariczocsi@yahoo.com}

\keywords{Mills' ratio, Reciprocally convex (concave) functions,
Monotone form of l'Hospital's rule, Statistical distributions,
Completely monotonic functions, Stieltjes transform.}

\begin{document}
\maketitle

\begin{center}
\footnotesize{Dedicated to my children Bor\'oka and Kopp\'any}
\end{center}

\begin{abstract}
This note contains sufficient conditions for the probability density
function of an arbitrary continuous univariate distribution,
supported on $(0,\infty),$ such that the corresponding Mills ratio
to be reciprocally convex (concave). To illustrate the applications
of the main results, the reciprocal convexity (concavity) of Mills
ratio of the gamma distribution is discussed in details.
\end{abstract}

\section{\bf Introduction}

By definition (see \cite{merkle}) a function
$f:(0,\infty)\to\mathbb{R}$ is said to be (strictly) reciprocally
convex if $x\mapsto f(x)$ is (strictly) concave and $x\mapsto
f(1/x)$ is (strictly) convex on $(0,\infty).$ Merkle \cite{merkle}
showed that $f$ is reciprocally convex if and only if for all
$x,y>0$ we have
\begin{equation}\label{eq1}f\left(\frac{2xy}{x+y}\right)\leq\frac{f(x)+f(y)}{2}\leq
f\left(\frac{x+y}{2}\right)\leq
\frac{xf(x)+yf(y)}{x+y}.\end{equation} We note here that in fact the
third inequality follows from the fact that the function $x\mapsto
f(1/x)$ is convex on $(0,\infty)$ if and only if $x\mapsto xf(x)$ is
convex on $(0,\infty).$ In what follows, similarly as in
\cite{merkle}, a function $g:(0,\infty)\to\mathbb{R}$ is said to be
(strictly) reciprocally concave if and only if $-g$ is (strictly)
reciprocally convex, i.e. if $x\mapsto g(x)$ is (strictly) convex
and $x\mapsto g(1/x)$ is (strictly) concave on $(0,\infty).$ Observe
that if $f$ is differentiable, then $x\mapsto f(1/x)$ is (strictly)
convex (concave) on $(0,\infty)$ if and only if $x\mapsto x^2f'(x)$
is (strictly) increasing (decreasing) on $(0,\infty).$

As it was shown by Merkle \cite{merkle}, reciprocally convex
functions defined on $(0,\infty)$ have a number of interesting
properties: they are increasing on $(0,\infty)$ or have a constant
value on $(0,\infty),$ they have a continuous derivative on
$(0,\infty)$ and they generate a sequence of quasi-arithmetic means,
with the first one between harmonic and arithmetic mean and others
above the arithmetic mean. Some examples of reciprocally convex
functions related to the Euler gamma function were given in
\cite{merkle}.

By definition (see \cite{widder}) a function
$f:(0,\infty)\to\mathbb{R}$ is said to be completely monotonic, if
$f$ has derivatives of all orders and satisfies
$$(-1)^nf^{(n)}(x)\geq 0$$ for all $x>0$ and
$n\in\{0,1,\dots\}.$ Note that strict inequality always holds above
unless $f$ is constant. It is known (Bernstein's Theorem) that $f$
is completely monotonic if and only if \cite[p. 161]{widder}
$$f(x)=\int_0^{\infty}e^{-xt}\doo\nu(t),$$
where $\nu$ is a nonnegative measure on $[0,\infty)$ such that the
integral converges for all $x>0.$ An important subclass of
completely monotonic functions consists of the Stieltjes transforms
defined as the class of functions $g:(0,\infty)\to\mathbb{R}$ of the
form
$$g(x)=\alpha+\int_0^{\infty}\frac{\doo\nu(t)}{x+t},$$
where $\alpha\geq0$ and $\nu$ is a nonnegative measure on
$[0,\infty)$ such that the integral converges for all $x>0.$

It was pointed out in \cite{merkle} that if a function
$h:(0,\infty)\to\mathbb{R}$ is a Stieltjes transform, then $-h$ is
reciprocally convex, i.e. $h$ is reciprocally concave. We note that
some known reciprocally concave functions comes from probability
theory. For example, the Mills ratio of the standard normal
distribution is a reciprocally concave function. For this let us see
some basics. The probability density function
$\varphi:\mathbb{R}\rightarrow(0,\infty),$ the cumulative
distribution function $\Phi:\mathbb{R}\to (0,1)$ and the reliability
function $\overline{\Phi}:\mathbb{R}\rightarrow(0,1)$ of the
standard normal law, are defined by
$$\varphi(x)=\frac{1}{\sqrt{2\pi}}e^{-x^2/2},$$
$$\Phi(x)=\int_{-\infty}^x\varphi(t)\dt$$
and $$\overline{\Phi}(x)=1-\Phi(x)=\int_{x}^{\infty}\varphi(t)\dt.$$
The function $m:\mathbb{R}\rightarrow(0,\infty),$ defined by
$$m(x)=\frac{\overline{\Phi}(x)}{\varphi(x)}=e^{x^2/2}\int_x^{\infty}e^{-t^2/2}\dt=\int_0^{\infty}e^{-xt}e^{-t^2/2}\dt,$$
is known in literature as Mills' ratio \cite[Sect. 2.26]{mitri} of
the standard normal law, while its reciprocal $r=1/m,$ defined by
$r(x)=1/m(x)=\varphi(x)/\overline{\Phi}(x),$ is the so-called
failure (hazard) rate. For Mills' ratio of other distributions, like
gamma distribution, we refer to \cite{mer} and to the references
therein.

It is well-known that Mills' ratio of the standard normal
distribution is convex and strictly decreasing on $\mathbb{R},$ at
the origin takes on the value $m(0)=\sqrt{\pi/2}.$ Moreover, it can
be shown (see \cite{mills2}) that $x\mapsto m'(x)/m(x)$ is strictly
increasing and $x\mapsto x^2m'(x)$ is strictly decreasing on
$(0,\infty).$ With other words, the Mills ratio of the standard
normal law is strictly reciprocally concave on $(0,\infty).$ Some
other monotonicity properties and interesting functional
inequalities involving the Mills ratio of the standard normal
distribution can be found in \cite{mills2}. The following
complements the above mentioned results.

\begin{theorem}
Let $m$ be the Mills ratio of the standard normal law. Then the
function $x\mapsto m(\sqrt{x})/\sqrt{x}$ is a Stieltjes transform
and consequently it is strictly completely monotonic and strictly
reciprocally concave on $(0,\infty).$ In particular, if $x,y>0,$
then the following chain of inequalities holds
\begin{align*}
\sqrt{\frac{x+y}{2xy}}&m\left(\sqrt{\frac{2xy}{x+y}}\right)\geq\frac{\sqrt{y}m(\sqrt{x})+\sqrt{x}m(\sqrt{y})}{2\sqrt{xy}}
\\&\geq\sqrt{\frac{2}{x+y}}m\left(\sqrt{\frac{x+y}{2}}\right)\geq\frac{\sqrt{x}m(\sqrt{x})+\sqrt{y}m(\sqrt{y})}{x+y}.
\end{align*}
In each of the above inequalities equality holds if and only if
$x=y.$
\end{theorem}

\begin{proof}
For $x>0$ the Mills of the standard normal distribution can be
represented as \cite[p. 145]{kendall}
$$m(x)=\int_{-\infty}^{\infty}\frac{x}{x^2+t^2}\varphi(t)\dt=2\int_{0}^{\infty}\frac{x}{x^2+t^2}\varphi(t)\dt.$$
From this we obtain that
$$\frac{m(\sqrt{x})}{\sqrt{x}}=\frac{1}{\sqrt{2\pi}}\int_0^{\infty}\frac{1}{x+s}\frac{e^{-s/2}}{\sqrt{s}}\ds,$$
which shows that the function $x\mapsto m(\sqrt{x})/\sqrt{x}$ is in
fact a Stieltjes transform and owing to Merkle \cite[p. 217]{merkle}
this implies that the function $x\mapsto -m(\sqrt{x})/\sqrt{x}$ is
reciprocally convex on $(0,\infty),$ i.e. the function $x\mapsto
m(\sqrt{x})/\sqrt{x}$ is reciprocally concave on $(0,\infty).$ The
rest of the proof follows easily from \eqref{eq1}. We note that the
strictly complete monotonicity of the function $x\mapsto
m(\sqrt{x})/\sqrt{x}$ can be proved also by using the properties of
completely monotonic functions. Mills ratio $m$ of the standard
normal distribution is in fact a Laplace transform and consequently
it is strictly completely monotonic (see \cite{mills2}). On the
other hand, it is known (see \cite{widder}) that if $u$ is strictly
completely monotonic and $v$ is nonnegative with a strictly
completely monotone derivative, then the composite function $u\circ
v$ is also strictly completely monotonic. Now, since the function
$m$ is strictly completely monotonic on $(0,\infty)$ and $x\mapsto
2(\sqrt{x})'=1/\sqrt{x}$ is strictly completely monotonic on
$(0,\infty),$ we obtain that $x\mapsto m(\sqrt{x})$ is also strictly
completely monotonic on $(0,\infty).$ Finally, by using the fact
that the product of completely monotonic functions is also
completely monotonic, the function $x\mapsto m(\sqrt{x})/\sqrt{x}$
is indeed strictly completely monotonic on $(0,\infty).$
\end{proof}

Now, since the Mills ratio of the standard normal distribution is
reciprocally concave a natural question which arises here is the
following: under which conditions does the Mills ratio of an
arbitrary continuous univariate distribution, having support
$(0,\infty),$ will be reciprocally convex (concave)? The goal of
this paper is to find some sufficient conditions for the probability
density function of an arbitrary continuous univariate distribution,
supported on the semi-infinite interval $(0,\infty),$ such that the
corresponding Mills ratio to be reciprocally convex (concave). The
main result of this paper, namely Theorem \ref{th2} in section 2, is
based on some recent results of the author \cite{mills} and
complement naturally the results from \cite{mills2,mills}. To
illustrate the application of the main result, the Mills ratio of
the gamma distribution is discussed in details in section 3.

We note that although the reciprocal convexity (concavity) of Mills
ratio is interesting in his own right, the convexity of the Mills
ratio of continuous distributions has important applications in
monopoly theory, especially in static pricing problems. For
characterizations of the existence or uniqueness of global
maximizers we refer to \cite{berg} and to the references therein.
Another application can be found in \cite{mart}, where the convexity
of Mills ratio is used to show that the price is a sub-martingale.

\section{\bf Reciprocal convexity (concavity) of Mills ratio}

In this section our aim is to find some sufficient conditions for
the probability density function such that the corresponding Mills
ratio to be reciprocally convex (concave). As in \cite{mills} the
proof is based on the monotone form of l'Hospital's rule \cite[Lemma
2.2]{anderson}.

\begin{theorem}\label{th2}
Let $f:(0,\infty)\to (0,\infty)$ be a probability density function
and let $\omega:(0,\infty)\to\mathbb{R},$ defined by
$\omega(x)=f'(x)/f(x),$ be the logarithmic derivative of $f.$ Let
also $\overline{F}:(0,\infty)\to(0,1),$ defined by
$\overline{F}(x)=\int_x^{\infty}f(t)\dt,$ be the survival function
and $m:(0,\infty)\to(0,\infty),$ defined by
$m(x)=\overline{F}(x)/f(x),$ be the corresponding Mills ratio. Then
the following assertions are true:
\begin{enumerate}
\item[{(a)}] If $f(x)/\omega(x)\to0$ as $x\to\infty,$ $\omega'/\omega^2$ is (strictly) decreasing
(increasing) on $(0,\infty)$ and the function
$$x\mapsto \frac{x^3\omega'(x)}{x\omega^2(x)-x\omega'(x)-2\omega(x)}$$
is (strictly) increasing (decreasing) on $(0,\infty),$ then Mills
ratio $m$ is (strictly) reciprocally convex (concave) on
$(0,\infty).$
\item[{(b)}] If $xf(x)/(1-x\omega(x))\to0,$ $f(x)/\omega(x)\to0$ as $x\to\infty,$ $\omega'/\omega^2$ is (strictly) decreasing
(increasing) on $(0,\infty)$ and the function
$$x\mapsto \frac{x^2\omega'(x)-x\omega(x)+2}{x\omega^2(x)-x\omega'(x)-2\omega(x)}$$
is (strictly) increasing (decreasing) on $(0,\infty)$, then Mills
ratio $m$ is (strictly) reciprocally convex (concave) on
$(0,\infty).$
\end{enumerate}
\end{theorem}
\begin{proof}
(a) By definition Mills ratio $m$ is (strictly) reciprocally convex
(concave) if $m$ is (strictly) concave (convex) and $x\mapsto
m(1/x)$ is (strictly) convex (concave). It is known (see
\cite[Theorem 2]{mills}) that if $f(x)/\omega(x)$ tends to zero as
$x$ tends to infinity and the function $\omega'/\omega^2$ is
(strictly) increasing (decreasing), then $m$ is (strictly) convex
(concave). Thus, we just need to find conditions for the (strict)
convexity (concavity) of the function $x\mapsto m(1/x).$ This
function is (strictly) convex (concave) on $(0,\infty)$ if and only
if the function $x\mapsto x^2m'(x)$ is (strictly) increasing
(decreasing) on $(0,\infty).$ On the other hand, observe that Mills
ratio $m$ satisfies the differential equation $$m'(x) =
-\omega(x)m(x) - 1.$$ Thus, by using the monotone form of
l'Hospital's rule (see \cite[Lemma 2.2]{anderson}) to prove that the
function \begin{align*}x\mapsto
x^2m'(x)&=-\frac{\left(\overline{F}(x)+f(x)/\omega(x)\right)-\lim\limits_{x\to\infty}\left(\overline{F}(x)+f(x)/\omega(x)\right)}
{f(x)/(x^2\omega(x))-\lim\limits_{x\to\infty}f(x)/(x^2\omega(x))}\\&=-\frac{\overline{F}(x)+f(x)/\omega(x)}{f(x)/(x^2\omega(x))}\end{align*}
is (strictly) increasing (decreasing) on $(0,\infty)$ it is enough
to show that
$$x\mapsto -\frac{\left(\overline{F}(x)+f(x)/\omega(x)\right)'}{\left(f(x)/(x^2\omega(x))\right)'}=\frac{x^3\omega'(x)}{x\omega^2(x)-x\omega'(x)-2\omega(x)}$$
is (strictly) increasing (decreasing) on $(0,\infty).$

(b) Observe that according to \cite[Lemma 2.2]{merkle} the function
$x\mapsto m(1/x)$ is (strictly)  convex (concave) if and only if
$x\mapsto xm(x)$ is (strictly) convex (concave) on $(0,\infty).$
Now, by using the monotone form of l'Hospital's rule (see
\cite[Lemma 2.2]{anderson}) the function
\begin{align*}
x\mapsto
\left(xm(x)\right)'&=m(x)-x-x\omega(x)m(x)=\frac{\overline{F}(x)-xf(x)/(1-x\omega(x))}{f(x)/(1-x\omega(x))}\\&=
\frac{\left(\overline{F}(x)-xf(x)/(1-x\omega(x))\right)-\lim\limits_{x\to\infty}\left(\overline{F}(x)-xf(x)/(1-x\omega(x))\right)}
{\left(f(x)/(1-x\omega(x))\right)-\lim\limits_{x\to\infty}\left(f(x)/(1-x\omega(x))\right)}
\end{align*}
is (strictly) increasing (decreasing) on $(0,\infty)$ if the
function
$$x\mapsto \frac{\left(\overline{F}(x)-xf(x)/(1-x\omega(x))\right)'}
{\left(f(x)/(1-x\omega(x))\right)'}=\frac{x^2\omega'(x)-x\omega(x)+2}{x\omega^2(x)-x\omega'(x)-2\omega(x)}$$
is (strictly) increasing (decreasing) on $(0,\infty).$ Notice that we used the fact that if $xf(x)/(1-x\omega(x))\to0$ as $x\to\infty,$ then
$f(x)/(1-x\omega(x))\to0$ as $x\to\infty.$
\end{proof}

We note here that the reciprocal concavity of the Mills ratio of the
standard normal distribution can be verified easily by using part
(a) or part (b) of Theorem \ref{th2}. More precisely, in the case of
the standard normal distribution we have $\omega(x)=-x,$
$\omega'(x)=-1.$ Consequently
$\varphi(x)/\omega(x)=-\varphi(x)/x\to0$ as $x\to\infty,$ the
function $x\mapsto \omega'(x)/\omega^2(x)=-1/x^2$ is strictly
increasing and
$$x\mapsto \frac{x^3\omega'(x)}{x\omega^2(x)-x\omega'(x)-2\omega(x)}=-\frac{x^2}{x^2+3}$$
is strictly decreasing on $(0,\infty).$ This is turn implies that by
using part (a) of Theorem \ref{th2} the Mills ratio of the standard
normal distribution is strictly reciprocally concave on
$(0,\infty).$

Similarly, since $\varphi(x)/(1+x^2)\to 0,$
$x\varphi(x)/(1+x^2)\to0,$ $-\varphi(x)/x\to0$ as $x\to\infty,$ the
function $x\mapsto \omega'(x)/\omega^2(x)=-1/x^2$ is strictly
increasing and
$$x\mapsto
\frac{x^2\omega'(x)-x\omega(x)+2}{x\omega^2(x)-x\omega'(x)-2\omega(x)}=\frac{2}{x^3+3x}$$
is strictly decreasing on $(0,\infty),$ part (b) of Theorem
\ref{th2} also implies that the Mills ratio of the standard normal
distribution is strictly reciprocally concave on $(0,\infty).$

Thus, Theorem \ref{th2} in fact generalizes some of the main results
of \cite{mills2}.

\section{\bf Reciprocal convexity (concavity) of Mills ratio of the gamma distribution}

The gamma distribution has support $(0,\infty),$ probability density
function, cumulative distribution function and survival function as
follows
$$f(x)=f(x;\alpha)=\frac{x^{\alpha-1}e^{-x}}{\Gamma(\alpha)},$$
$$F(x)=F(x;\alpha)=\frac{\gamma(\alpha,x)}{\Gamma(\alpha)}=\frac{1}{\Gamma(\alpha)}\int_0^xt^{\alpha-1}e^{-t}\dt$$
and
$$\overline{F}(x)=\overline{F}(x;\alpha)=\frac{\Gamma(\alpha,x)}{\Gamma(\alpha)}=\frac{1}{\Gamma(\alpha)}\int_x^{\infty}t^{\alpha-1}e^{-t}\dt,$$
where $\Gamma$ is the Euler gamma function, $\gamma(\cdot,\cdot)$
and $\Gamma(\cdot,\cdot)$ denote the lower and upper incomplete
gamma functions, and $\alpha>0$ is the shape parameter. As we can
see below, {\em the Mills ratio of the gamma distribution}
$m:(0,\infty)\to(0,\infty),$ defined by
$$m(x)=m(x;\alpha)=\frac{\Gamma(\alpha,x)}{x^{\alpha-1}e^{-x}},$$
{\em is reciprocally convex on $(0,\infty)$ for all $0<\alpha\leq1$
and reciprocally concave on $(0,\infty)$ for all
$1\leq\alpha\leq2.$} In \cite{mills} it was proved that if
$\alpha\geq1,$ then the Mills ratio $m$ is decreasing and
log-convex, and consequently convex on $(0,\infty).$ We note that
the convexity of Mills ratio of the gamma distribution actually can
be verified directly (see \cite{mart}), since
$$m(x)=\int_x^{\infty}\left(\frac{t}{x}\right)^{\alpha-1}e^{x-t}dt=\int_1^{\infty}xu^{\alpha-1}e^{(1-u)x}\du$$
and
\begin{align*}m'(x)&=\int_1^{\infty}\left((\alpha-1)u^{\alpha-2}\right)\left((1-u)e^{(1-u)x}\right)\du\\&
=\int_1^{\infty}u^{\alpha-1}e^{(1-u)x}du+\int_1^{\infty}xu^{\alpha-1}(1-u)e^{(1-u)x}\du,\end{align*}
where the last equality follows from integration by parts. From this
we clearly have that
$$m''(x)=\int_1^{\infty}(\alpha-1)(1-u)^2u^{\alpha-2}e^{(1-u)x}du$$
and consequently $m$ is convex on $(0,\infty)$ if $\alpha\geq1$ and
is concave on $(0,\infty)$ if $0<\alpha\leq1.$ The concavity of the
function $m$ can be verified also by using \cite[Theorem 2]{mills}.
Namely, if let $\omega(x)=f'(x)/f(x)=(\alpha-1)/x-1,$ then
$f(x)/\omega(x)$ tends to zero as $x$ tends to infinity and the
function $x\mapsto \omega'(x)/\omega^2(x)=(1-\alpha)/(\alpha-1-x)^2$
is decreasing on $(0,\infty)$ for all $0<\alpha\leq1.$ Consequently
in view of \cite[Theorem 2]{mills} $m$ is indeed concave on
$(0,\infty)$ for all $0<\alpha\leq1.$

Now let us focus on the reciprocal convexity (concavity) of the
Mills ratio of gamma distribution. Since
$$\frac{x^3\omega'(x)}{x\omega^2(x)-x\omega'(x)-2\omega(x)}=\frac{(1-\alpha)x^2}{(\alpha-1-x)^2+2x+1-\alpha},$$
we obtain that
$$\left(\frac{x^3\omega'(x)}{x\omega^2(x)-x\omega'(x)-2\omega(x)}\right)'=\frac{2(\alpha-1)(\alpha-2)(x^2+(1-\alpha)x)}{((\alpha-1-x)^2+2x+1-\alpha)^2}.$$
This last expression is clearly positive if $0<\alpha\leq1$ and
$x>0,$ and thus, by using part (a) of Theorem \ref{th2} we conclude
that Mills ratio $m$ is reciprocally convex on $(0,\infty)$ for all
$0<\alpha\leq1.$

Similarly, since
$$\frac{x^2\omega'(x)-x\omega(x)+2}{x\omega^2(x)-x\omega'(x)-2\omega(x)}=\frac{x^2+2(2-\alpha)x}{x^2+2(2-\alpha)x+(\alpha-1)(\alpha-2)},$$
we get
\begin{align*}\left(\frac{x^2\omega'(x)-x\omega(x)+2}{x\omega^2(x)-x\omega'(x)-2\omega(x)}\right)'=
\frac{2(\alpha-1)(\alpha-2)(x+2-\alpha)}{(x^2+2(2-\alpha)x+(\alpha-1)(\alpha-2))^2}\end{align*}
and this is negative if $1\leq \alpha\leq 2$ and $x>0.$
Consequently, by using part (b) of Theorem \ref{th2} we get that the
Mills ratio of the gamma distribution is indeed reciprocally concave
for $1\leq\alpha\leq2.$ Here we used that if $x$ tends to $\infty,$
then the expressions $f(x)/\omega(x)$ and
$xf(x)/(1-x\omega(x))$ tend to $0.$

Finally, we note that the convexity (concavity) of $x\mapsto m(1/x)$
can be verified also by using the integral representation of Mills
ratio of the gamma distribution. More precisely, if we rewrite
$m(x)$ as
$$m(x)=\int_0^{\infty}\left(1+\frac{u}{x}\right)^{\alpha-1}e^{-u}\du,$$
then
$$x^2m'(x)=-\int_0^{\infty}(\alpha-1)\left(1+\frac{u}{x}\right)^{\alpha-2}ue^{-u}\du$$
and
$$[x^2m'(x)]'=\int_0^{\infty}(\alpha-1)(\alpha-2)\left(1+\frac{u}{x}\right)^{\alpha-3}\frac{u^2}{x^2}e^{-u}\du.$$
This shows that $x\mapsto x^2m'(x)$ is decreasing on $(0,\infty)$ if
$1\leq\alpha\leq2$ and increasing on $(0,\infty)$ if $0<\alpha\leq1$
or $\alpha\geq2.$ Summarizing, the Mills ratio of the gamma
distribution is reciprocally convex on $(0,\infty)$ if
$0<\alpha\leq1$ and reciprocally concave on $(0,\infty)$ if
$1\leq\alpha\leq2.$ {\em When $\alpha>2$ the functions $x\mapsto
m(x)$ and $x\mapsto m(1/x)$ are convex on $(0,\infty),$ thus in this
case $m$ is nor reciprocally convex and neither reciprocally concave
on its support.}

\subsection*{Acknowledgments}

This work was supported by the J\'anos Bolyai Research Scholarship
of the Hungarian Academy of Sciences and by the Romanian National
Authority for Scientific Research CNCSIS-UEFISCSU, project number
PN-II-RU-PD\underline{ }388/2012. The author is grateful to the referee
for pointing out some errors in the manuscript and also to Dr. Ming
Hu from University of Toronto for bringing the works
\cite{berg,mart} to his attention and for giving him the motivation
of the study of the convexity of Mills ratio.

\end{document}